\documentclass[12pt]{iopart}

\usepackage{iopams, harvard, amsthm, amssymb}  
\usepackage{graphicx,xcolor}  

\usepackage{array,subfig}
\newcolumntype{L}{>{$}l<{$}} % text mode "l" in an "array"

\newtheorem{thm}{Theorem}

\renewcommand{\L}{\mathcal{L}}

\newcommand{\Lv}{L_{\bf v}}
\newcommand{\Lb}{L_{\bf b}}

\newcommand{\Lvp}{L_{{\bf v}'}}

\newcommand{\gammat}{\widetilde{\gamma}}

\begin{document}

\title[Field line winding of braided vector fields]{Field line winding of braided vector fields in tubular subdomains}
% ARY: I think the title "intrinsic winding" was more appropriate when we had the other part of the paper.

\author{Christopher B Prior and Anthony R Yeates}

\address{Department of Mathematical Sciences, Durham University, South Road, Durham, DH1 3LE, UK}
\ead{christopher.prior@durham.ac.uk}
\vspace{10pt}
\begin{indented}
\item[]November 2018
\end{indented}

\begin{abstract}
Braided vector fields on spatial subdomains homeomorphic to the cylinder play a crucial role in applications such as solar and plasma physics, relativistic astrophysics, fluid and vortex dynamics, elasticity, and bio-elasticity. Often the vector field's topology -- the entanglement of its field lines -- is non-trivial, and can play a significant role in the vector field's evolution. We present a complete topological characterisation of such vector fields (up to isotopy) using a quantity called field line winding. This measures the entanglement of each field line with all other field lines of the vector field, and may be defined for an arbitrary tubular subdomain by prescribing a minimally distorted coordinate system. We propose how to define such coordinates, and prove that the resulting field line winding distribution uniquely classifies the topology of a braided vector field. The field line winding is similar to the field line helicity considered previously for magnetic (solenoidal) fields, but is a more fundamental measure of the field line topology because it does not conflate linking information with field strength.
\end{abstract}

\submitto{\JPA}
%\maketitle

\section{Introduction}
The entanglement of vector field integral curves (field lines) in tubular subdomains homeomorphic to the cylinder has long been of wide interest.
%In this paper a tubular subdomain $\mathcal{M}\in\mathbb{R}^3$ is defined as one homeomorphic  to the cylinder (with Lipshitz continuity), and we consider vector fields on tubular domains whose integral curves (field lines) flow through $\mathcal{M}$ (as indicated in Figure \ref{domainex}a).
For example, in stellar interiors, twisted bundles of magnetic field lines -- known as magnetic flux ropes -- are potentially important for the generation of large scale magnetic fields \cite{childress1995stretch,moffatt1985topological,nordlund1992dynamo,bao2010complex}. In stellar atmospheres such as the Sun's corona, twisted and braided magnetic fields play a crucial role in dynamic phenomena such as coronal heating, jet formation, and coronal mass ejections \cite{rust1996evidence,torok2005confined,wilmot2015overview,prior2016emergence,yeates2016global}. Another area where the dynamic entanglement of magnetic field is  crucial is the formation of relativistic jets from black holes and neutron stars \cite{komissarov1999numerical,heinz2000jet,contopoulos2009invariant,prior2019observational}. In thin body elasticity, structures such as ropes, cables, and biopolymers are treated as thin elastic tubes which can be internally twisted and which often react to this twisting by forming looped and knotted structures \cite{goriely1998spontaneous,thompson2002supercoiling,grason2009braided,starostin2014theory,prior2016extended}. A particularly well known example is plectoneme formation, whereby an elastic tube submitted to increasing twisting becomes unstable and loops at its centre, forming a self-contacting loop. As the twist is further increased the number of loops increase and the tube becomes supercoiled. DNA is an example of a biopolymer which exhibits this supercoiling \citeaffixed{forth2008abrupt}{and where twisting can be applied by optical tweezer experiments, \textit{e.g.},}. Supercoiling acts as a mechanism for large DNA molecules to compactify in order to fit into their local environment \cite{mullinger1980packing}. In all of these varied systems, a crucial aspect is their inability to disentangle, at least completely. This often informs their physical behaviour, and means that quantifying entanglement is important for understanding the systems themselves.

Motivation for our work comes from the particular context of magnetised plasmas, where the standard tool for quantifying entanglement is a volume integral called magnetic helicity, $H$ \cite{woltjer1958,berger1984topological}, analogous to a similar quantity in fluid dynamics \cite{moreau1961,moffatt1969degree,arnold1992topological}. For a ``closed'' magnetic field whose field line curves are tangent to the boundary, the value of this integral is invariant under an ideal evolution because field lines cannot reconnect and change their topology. As a result, having non-zero $H$ constrains the amount of magnetic energy that can be released \cite{arnold1992topological}. In such a magnetically-closed volume, it has long been known that $H$ may be written as an average of the Gauss linking integral between all pairs of magnetic field line curves \cite{pohl1968self,moffatt1969degree,cantarella2010new}. Recently, \citeasnoun{prior2014helicity} showed that the $H$ admits an analogous interpretation for so-called ``braided'' magnetic fields, where magnetic field lines connect between two planar boundaries rather than being tangent. In this case, $H$ is gauge dependent, but the authors showed that there is a particular gauge, the ``winding gauge'', in which $H$ is the equal to the average winding number between all pairs of field line curves. This extends the topological interpretation of helicity beyond the restricted case of tangent fields, which are not directly relevant for applications such as stellar atmospheres.

The main limitation of the winding-gauge helicity of \citeasnoun{prior2014helicity} is that it is defined only for domains foliated by planes, for which the standard definition of winding number can be used. However, in many contexts it would be useful to consider tubular domains of more general shape. One example would be the concentrated magnetic flux ropes that form during flux emergence into the solar atmosphere. Such flux ropes typically occupy curved and distorted domains that cannot be represented as foliations of parallel planes \cite{longcope2008}. Yet to characterise the internal entanglement of such structures it is necessary to isolate them from the surrounding magnetic field. Recent work has suggested that these structures may be more stable if they have a complex, ``braided'' internal field-line topology, rather than a uniformly twisted structure \cite{prior2016twisted}. Even in stable loops, \citeasnoun{wilmot2011} showed that different internal braiding patterns can lead to significantly different internal heating. It is important to note that the \citeasnoun{prior2014helicity} definition of winding-gauge helicity will fail not only if the tubular domain as a whole has a curved shape, but also if its end ``caps'' are not flat. Recently, \citeasnoun{prior2019interpreting} have shown that accounting for such wrinkling of the Sun's surface can significantly affect the measured input of magnetic helicity into the solar atmosphere. Observational studies of such helicity injection are important because they can potentially be used try to predict the ensuing behaviour of the atmospheric magnetic field \cite{romano2011solar,dalmasse2014photospheric,gopalswamy2017new}, as well as providing insight into the unobservable magnetic structure in the solar interior. Thus there is a need to consider winding numbers on more general domains, so as to quantify the entanglement of vector fields on those domains. This is the main aim of this paper.

To fully quantify the entanglement of a braided vector field, it is not sufficient to consider only a single, global integral like magnetic helicity, $H$. To give an extreme example, since $H$ is a signed quantity, one can have a topologically non-trivial vector field with zero total $H$, such as the Borromean rings \cite{delsordo2010} or the equivalent magnetic braid \cite{wilmot2015overview}. Moreover, in real systems that are not perfectly ideal, it is often the case that there are substantial local changes in entanglement, even when the total $H$ (the average entanglement) is conserved to a good approximation \cite{taylor1974,russell2015evolution}. A full understanding of the field line topology therefore requires the study of finer-grained invariants. As such, our main quantity of interest in this paper is not the analogue of $H$ (which would be the average winding among all pairs of field lines), but rather the analogue of field line helicity \cite{berger1988energy,yeates2013unique,aly2018,yeates2018relative,moraitis2019relative}. This we call the field line winding, as it is defined for each field line and measures the average winding of this field line with all others.

Our new measure of field line winding allows not only for more general shapes of domain but also for more general classes of vector field. We still assume that the vector field is ``braided'' (defined precisely in Section \ref{sec:ass}), but unlike \citeasnoun{prior2014helicity} we do not require it to be solenoidal (divergence free). As will be discussed in Section \ref{sec:helicity}, this amounts to removing the field-strength information from the winding measure. Magnetic helicity, for example, has units of magnetic flux squared, and is effectively a confluence of both topological and strength information. As shown by \citeasnoun{prior2014helicity}, a given set of field line curves can have different helicity depending on the field strength. In this paper, we focus purely on how to uniquely describe the field line topology, without regard to the strength.  This unweighted, purely topological measure applies equally to non-solenoidal fields, whereas the magnetic helicity, for example, is an ideal invariant only thanks to the solenoidal nature of the magnetic field. 

The paper is organised as follows. After stating our assumptions on the vector field and domain in Section \ref{sec:ass}, we give the general definition of winding numbers and field line winding in Section \ref{sec:flw}. These definitions rely on establishing a least distorted vector field on the domain, whose field lines have no mutual winding. This is addressed in Section \ref{sec:coords}. In Section \ref{sec:top}, we prove that the field line topology is completely determined by the distribution of field line winding. This is an extension of an earlier completeness result that applied only to a more restricted class of solenoidal vector fields \cite{yeates2013unique,yeates2014complete,prior2018quantifying}. As well as extending the class of domains and vector fields considered, we prove that the field line winding measure uniquely determines not only the field line mapping (as in the earlier papers) but also whether or not the two vector fields can be linked by an ideal evolution (an end-vanishing isotopy). Thus we also strengthen the earlier result for magnetic fields. The relation of our field line winding invariant to magnetic helicity is discussed in Section \ref{sec:helicity}, before concluding in Section \ref{sec:conclusion}.

\section{Assumptions and notation} \label{sec:ass}

\begin{figure}
\centering
\includegraphics[width=0.8\textwidth]{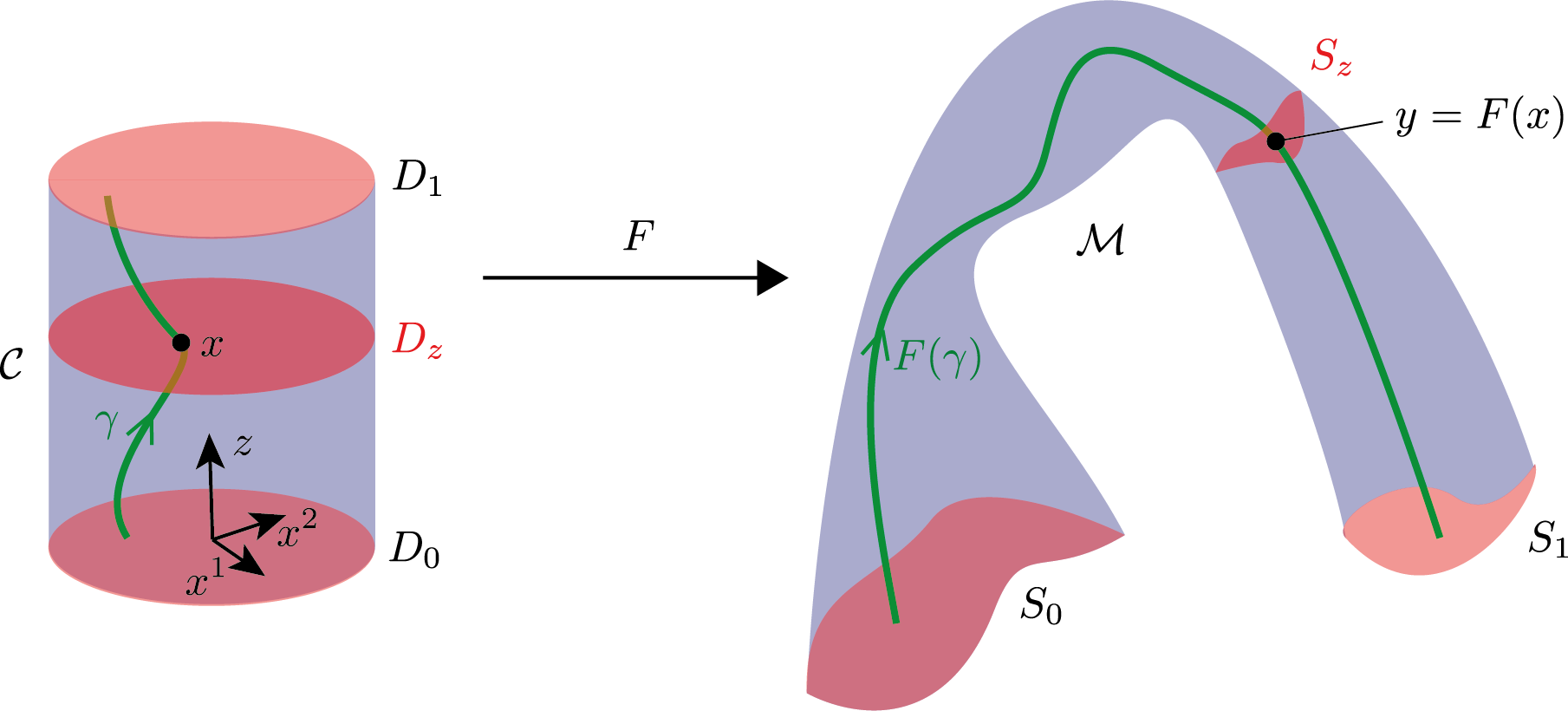}
\caption{\label{mfigure}A tubular subdomain $\mathcal{M}$ is an embedding of the unit cylinder in $\mathbb{R}^3$ given by a homeomorphism $F:\mathcal{C}\to\mathbb{R}^3$. The discs $D_z$ of constant $z$ in the cylinder $\mathcal{C}$ map to distorted surfaces $S_z$ that foliate $\mathcal{M}$. Shown in green are a field line $F(\gamma)\in\mathcal{M}$ and its preimage curve $\gamma\in\mathcal{C}$. A crucial part of this study will be to determine a minimally-distorted choice for the map $F$, within a given subdomain $\mathcal{M}$.
}
\end{figure}
\begin{figure}
\centering
\includegraphics[width=0.4\textwidth]{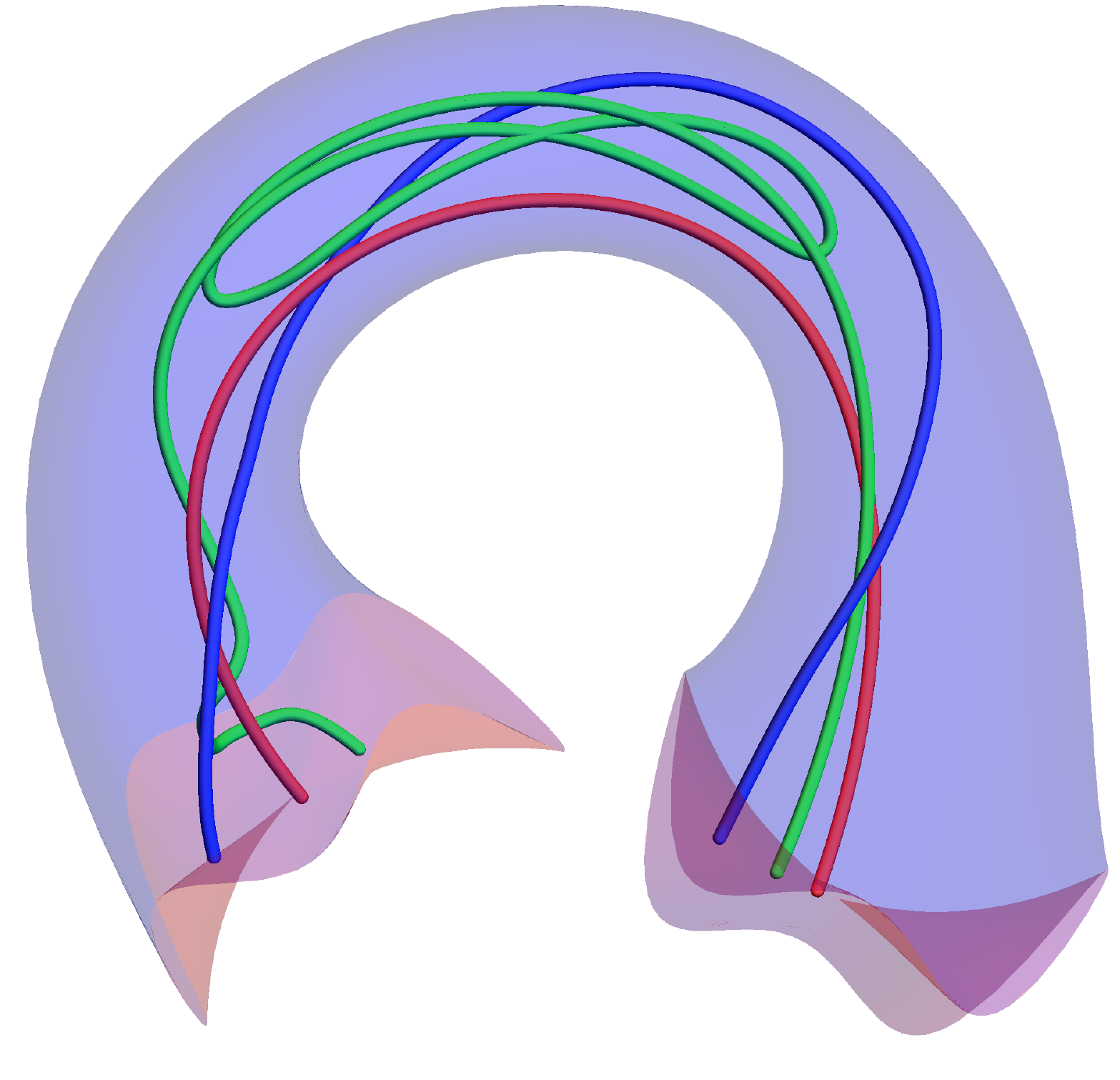}
\caption{\label{allowed}Three curves that could be field lines of a braided vector field as defined in Section \ref{sec:ass}. The green curve would not be monotone in $z$ but is still valid provided that the corresponding ${\bf v}$ can be deformed by isotopy to a field with field lines that are monotone in $z$.
}
\end{figure}

In this paper, we consider vector fields ${\bf v}$ defined on subdomains $\mathcal{M}\subset\mathbb{R}^3$. Both are restricted, in the following ways. 
The subdomain $\mathcal{M}$ is ``tubular'', meaning that it may be thought of as an embedding of the solid unit cylinder $\mathcal{C}$ in $\mathbb{R}^3$, determined by some homeomorphism $F:\mathcal{C}\to\mathbb{R}^3$, as in Figure \ref{mfigure}. This homeomorphism $F$ is by no means unique, and an important part of this paper will be to choose a suitable $F$. Note that the boundary $\partial\mathcal{M} = S_0 \cup S_1 \cup S_s$ is the union of two end caps $S_0, S_1$ and a side boundary $S_s$, and this composite boundary is assumed Lipschitz continuous. This means that $\mathcal{M}$ can be foliated by a set of surfaces $S_z,\, z \in[0,1]$. We will choose $F$ so that these surfaces are the images of the discs $D_z$ of constant $z$ in $\mathcal{C}$, where we define Cartesian coordinates on $\mathcal{C}$ as $(x^1,x^2,z)$. This notation recognizes the special role of the third coordinate in defining the foliation. In effect, $z$ becomes an axial coordinate for $\mathcal{M}$. Throughout this paper we will denote points in $\mathcal{C}$ by $x$ and points in $\mathcal{M}$ by $y$.

The vector field ${\bf v}$ is ``braided'', meaning that it is Lipschitz continuous, non-zero everywhere in $\mathcal{M}\cup\partial\mathcal{M}$, and satisfies the boundary conditions
\begin{equation}
\hat{\bf z}\cdot{\bf v} > 0 \quad \textrm{on $S_0$},\qquad
\hat{\bf z}\cdot{\bf v} > 0 \quad \textrm{on $S_1$},\qquad
\hat{\bf n}\cdot{\bf v} = 0 \quad \textrm{on $S_s$},
\end{equation}
where $\hat{\bf z}$ is a unit vector normal to $S_z$ pointing in the direction of increasing $z$ and $\hat{\bf n}$ is the unit normal to $S_s$. Thus  the integral curves/field lines of ${\bf v}$ are all rectifiable curves that connect from $S_0$ to $S_1$, as shown for the three example field lines in Figure \ref{allowed}. The final assumption on ${\bf v}$ is that its field lines can be deformed by isotopy to a vector field satisfying $\hat{\bf z}\cdot{\bf v} > 0$ on every $S_z$.  By an isotopy we mean a continuous set of homeomorphisms of the curves (which are automorphisms of the domain $\mathcal{M}$).
 For such fields, we can define a mapping from $S_ 0$ to $S_z$, or equivalently from $D_0$ to $D_z$ by following the field lines. These mappings are diffeomorphisms thanks to the Picard-Lindel\"of theorem \citeaffixed{arnolʹd2012geometrical}{see, \textit{e.g},}. In particular, we define a diffeomorphism $f(1)$ from $D_0$ to $D_1$ that will be used in Section \ref{sec:top}.
 
\section{Field line winding}\label{sec:flw}

Our goal in this paper is to describe the field-line topology of a braided vector field ${\bf v}$ in $\mathcal{M}$. In other words, properties of the field lines that remain unchanged under a continuous isotopy/deformation of ${\bf v}$. Since the field lines exit $\mathcal{M}$ through $S_0$ and $S_1$, we must qualify this statement: we consider properties that remain invariant under isotopies of ${\bf v}$ \emph{that vanish on $S_0$ and $S_1$}. These we call end-vanishing isotopies.

Our topological invariants are based on winding numbers between field lines. In this paper, we take advantage of the embedding of $\mathcal{M}$ and define the winding number between any two curves in $\mathcal{M}$ by calculating the winding number between their two preimage curves in $\mathcal{C}$. An alternative approach would be to work in $\mathcal{M}$ directly by choosing a metric structure. The advantage of our approach is that it allows us to define the angle in a simple and clear manner (\textit{i.e.}, on the cylinder), as well as simplifying the proof in Section \ref{sec:top}.
Figure \ref{fieldlinewindfigs}(a) shows field lines of an example vector field on a tubular domain $\mathcal{M}$, and Figure \ref{fieldlinewindfigs}(b) shows the preimages of these curves in $\mathcal{C}$ for a particular choice of the embedding map $F$.
How best to choose $F$ will be addressed in Section \ref{sec:coords}.

\begin{figure}
\subfloat[]{\includegraphics[width=3cm]{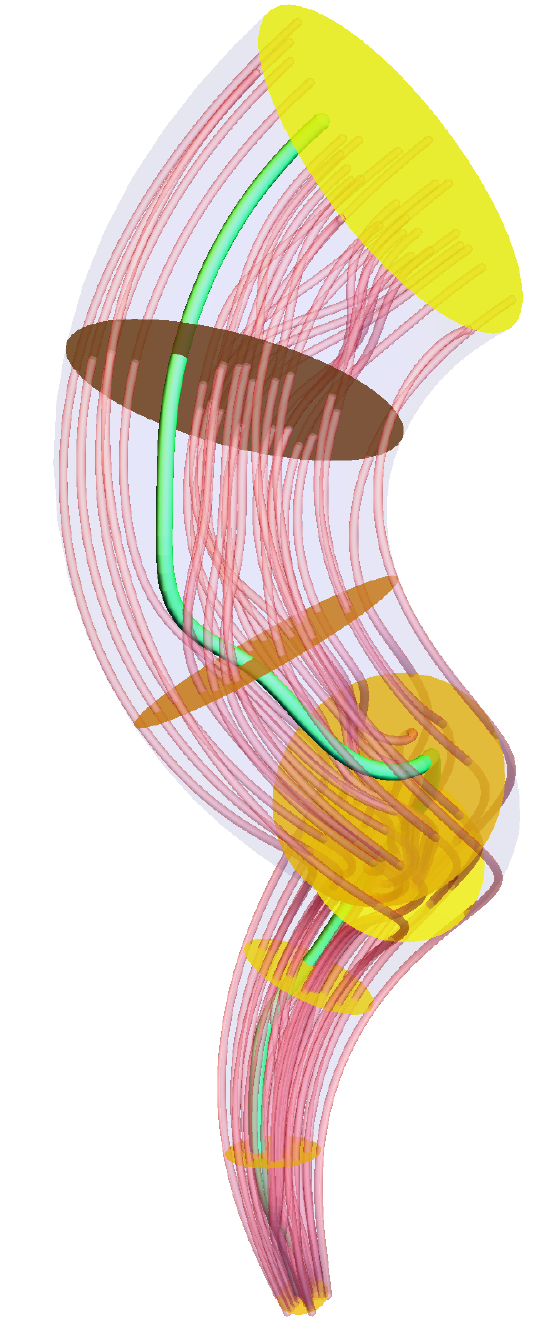}}\quad \quad \subfloat[]{\includegraphics[width=4cm]{undistortedFieldLines.pdf}}\quad \subfloat[]{\includegraphics[width=7cm]{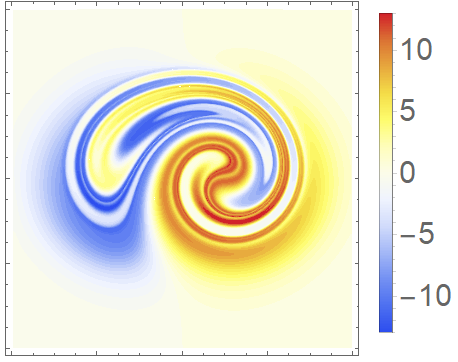}}.
\caption{\label{fieldlinewindfigs} Example illustrating the mapping of field line curves from a tubular domain $\mathcal{M}$ (a) to the cylinder $\mathcal{C}$ (b), along with the resulting distribution of field-line winding $\Lv$ in $D_0$ (c). The green curves in (a) and (b) show a particular field line $\gamma$ and its preimage. The interpretation of $\Lv(\gamma)$ is its average winding number with each of the other curves. Overall, this vector field shows significant entanglement due to the intermixing of positive and negative winding.}
\end{figure}

\subsection{Pairwise winding number}

\begin{figure}
    \centering
    \includegraphics[width=0.5\textwidth]{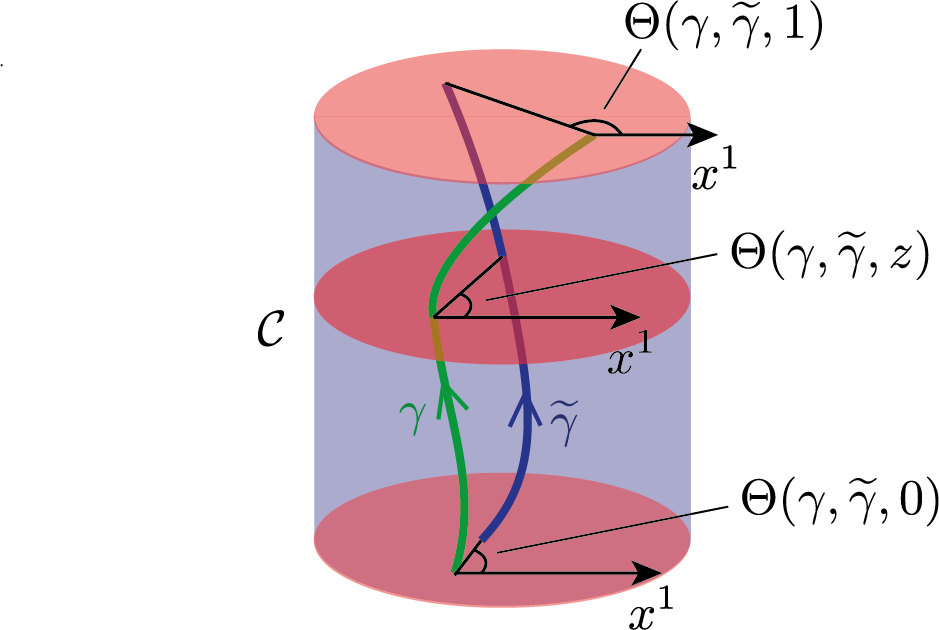}
    \caption{Definition of the angle $\Theta$ between two curves $\gamma$, $\gammat$ on $\mathcal{C}$.}
    \label{fig:Theta}
\end{figure}

Given two preimage curves $\gamma$, $\gammat$ in $\mathcal{C}$, the most basic topological invariant is their pairwise winding number. In any plane $D_z$, we can use the Cartesian components of the two curves $\gamma$ and $\gammat$ to define the ``angle'' between them,
\begin{equation}
\Theta(\gamma,\gammat,z) = \arctan\left(\frac{\gamma_2(z) - \gammat_2(z)}{\gamma_1(z) - \gammat_1(z)}\right),
\label{eqn:Theta}
\end{equation}
as shown in Figure \ref{fig:Theta}. The net change in this angle as we follow the curves from $z=0$ to $z=1$ is the pairwise winding number
\begin{eqnarray}
\L(\gamma,\gammat) &= \frac{1}{2\pi}\int_{0}^{1}\frac{\mathrm{d}}{\mathrm{d}z}\Theta(\gamma,\gammat, z)\,\mathrm{d}z \label{eqn:lorig}\\
&= \frac{1}{2\pi}\int_{0}^{1}\frac{(\gamma_1 - \gammat_1)\frac{\mathrm{d}}{\mathrm{d}z}(\gamma_2 - \gammat_2) - (\gamma_2 - \gammat_2)\frac{\mathrm{d}}{\mathrm{d}z}(\gamma_1 - \gammat_1)}{(\gamma_1-\gammat_1)^2 + (\gamma_2-\gammat_2)^2} \,\mathrm{d}z,
\label{eqn:lgg}
\end{eqnarray}
where primes denote differentiation by $z$. Alternatively, we may write
\begin{equation}
\L(\gamma,\gammat) = \frac{1}{2\pi}\Big(\Theta(\gamma,\gammat, 1) - \Theta(\gamma,\gammat, 0)\Big) + N(\gamma,\gammat),
\label{eqn:lN}
\end{equation}
where $N(\gamma,\gammat)$ counts the (signed) number of branch cut crossings  as we follow the curves in $z$. This makes clear that $\L(\gamma,\gammat)$ is invariant under any isotopy of the curves that does not move their end-points. In other words, it is a topological invariant.

\citeasnoun{berger2006writhe} showed that the definition of $\L(\gamma,\gammat)$ may be generalised to allow for non-monotonic curves like the green curve in Figure \ref{allowed}. Such a curve $\gamma$ is split into $n+1$ sections $\gamma^{(0)}, \ldots, \gamma^{(n)}$ using the $n$ turning points where $\mathrm{d}\gamma_z/\mathrm{d}z=0$. For each section, we define the indicator function
\begin{equation}
\sigma^{(i)} =
\left\{ \begin{array}{@{\kern2.5pt}lL}
    \hfill 1 & if $\mathrm{d}\gamma^{(i)}_z/\mathrm{d}z>0$,\\
    \hfill -1 & if $\mathrm{d}\gamma^{(i)}_z/\mathrm{d}z<0$,\\
     \hfill  0 & otherwise.
\end{array}\right.
\label{eqn:sigma}
\end{equation}
We split $\gammat$ and define $\widetilde{\sigma}^{(j)}$ in a similar way. Then the pairwise winding number is the sum
\begin{equation}
\L(\gamma,\gammat) = \sum_{i=0}^n\sum_{j=0}^{\widetilde{n}} \frac{\sigma^{(i)}\widetilde{\sigma}^{(j)}}{2\pi}\int_{z_{ij}^{\rm min}}^{z_{ij}^{\rm max}}\frac{\mathrm{d}}{\mathrm{d}z}\Theta(\gamma^{(i)}, \gammat^{(j)},z)\,\mathrm{d}z,
\label{eqn:l}
\end{equation}
where $[z_{ij}^{\rm min}, z_{ij}^{\rm max}]$ is the mutual range of $z$ values (if any) shared by the curve sections $\gamma^{(i)}$ and $\widetilde\gamma^{(j)}$. Once again, $\L(\gamma,\gammat)$ is invariant to any isotopy of the curves that fixes their endpoints. It reduces to (\ref{eqn:lorig}) if both curves have only a single section stretching from $z=0$ to $z=1$. (Although not needed here,  \citeasnoun{berger2006writhe} also showed that when the curves are closed, $\L(\gamma,\gammat)$ is equal to their Gauss linking integral, hence the notation ``$\L$''.)

\begin{figure}
    \centering
    \subfloat[]{\includegraphics[width=6cm]{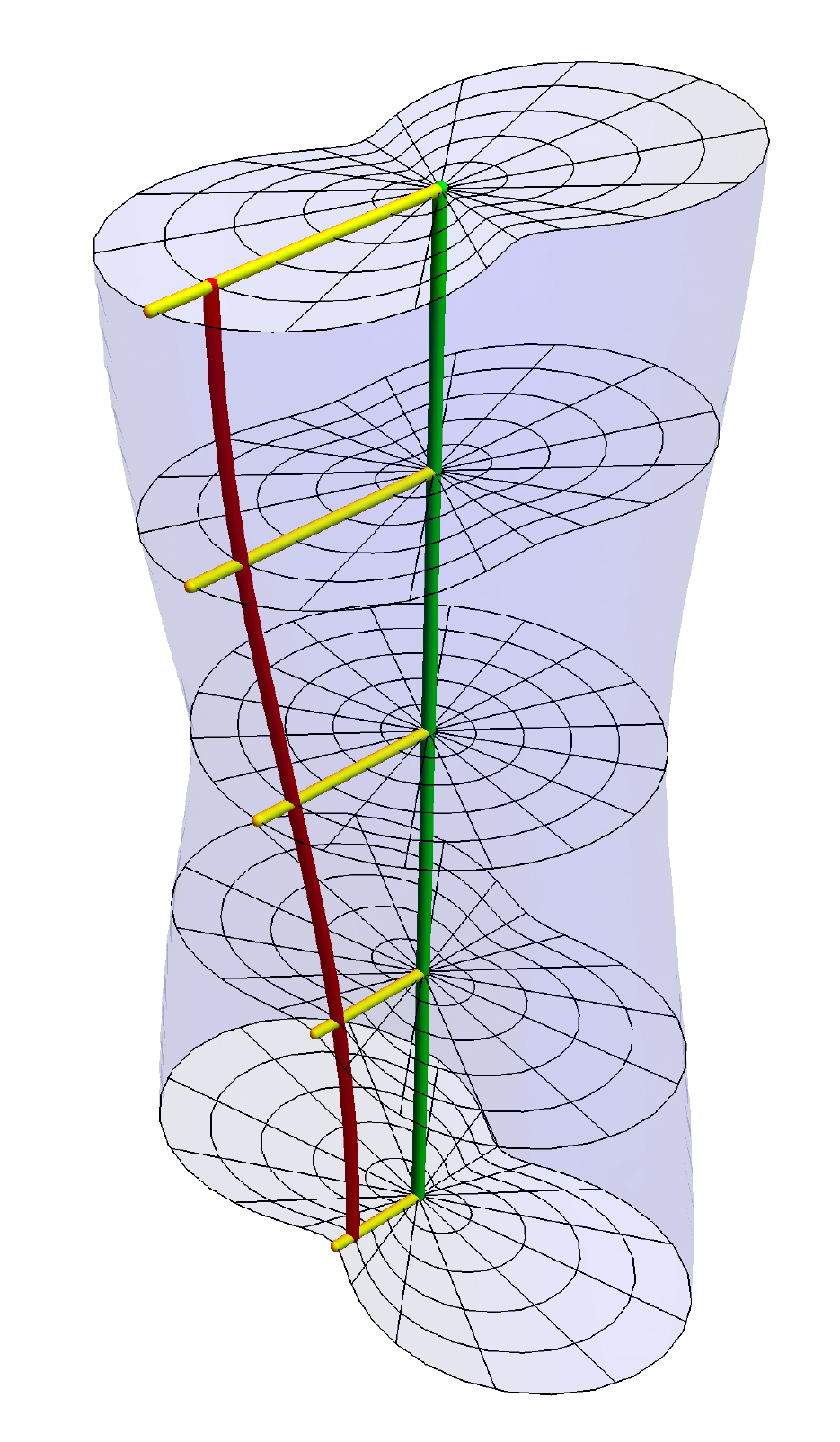}}\quad\subfloat[]{\includegraphics[width=6cm]{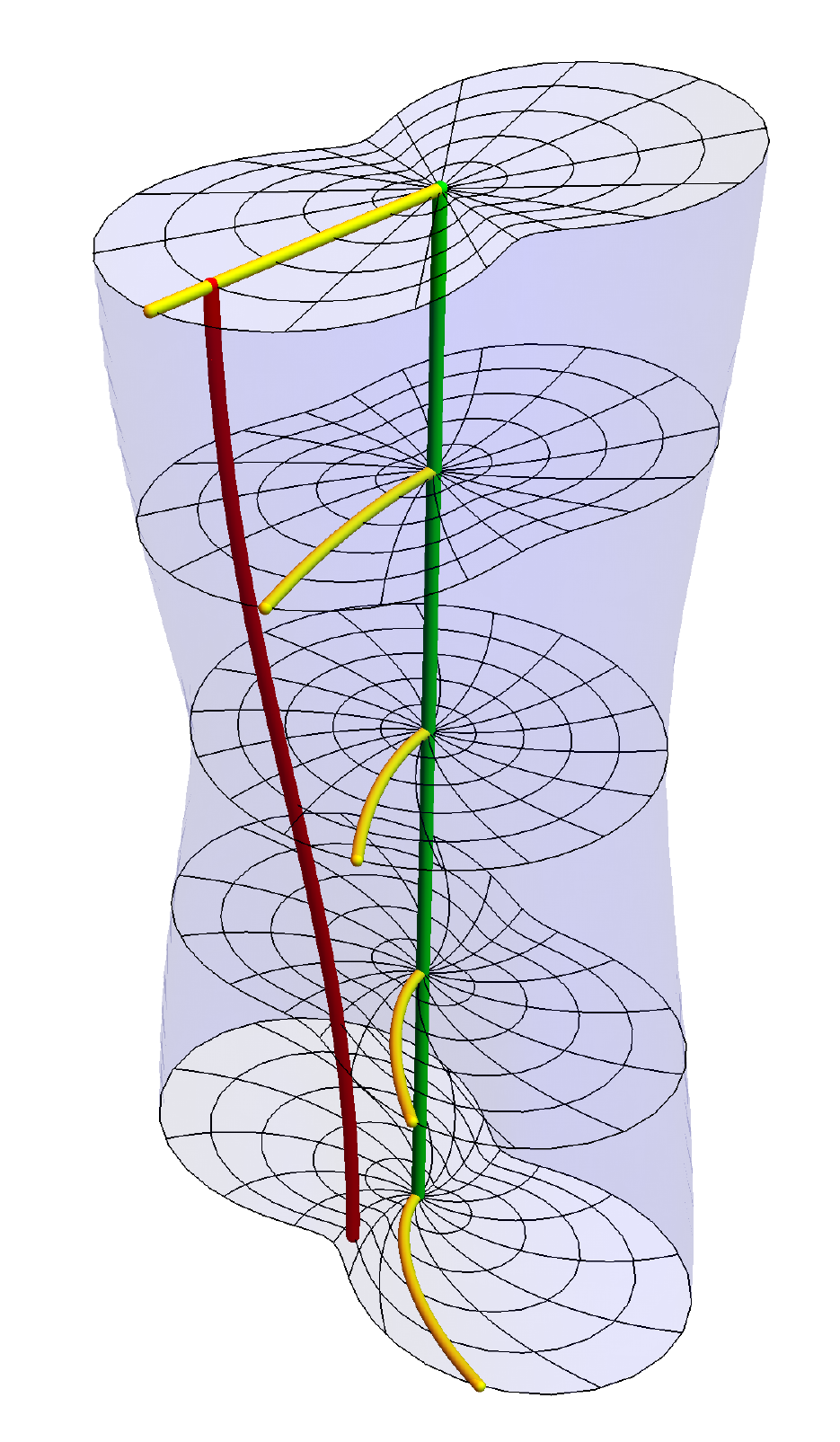}}
    \caption{How two different choices of mapping $F$ can lead to different pairwise winding number $\L(\gamma,\gammat)$. Here the red and green curves show $F(\gamma)$ and $F(\gammat)$, while yellow lines indicate the surface that $F(\gamma)$ should belong to in each coordinate system to give constant $\Theta(\gamma,\gammat,z)$. In (a) $F(\gamma)$ lies in this surface, so $\L(\gamma,\gammat)=0$, but in (b) it does not and $\L(\gamma,\gammat)\neq0$.}
    \label{goodbadcoords}
\end{figure}

It is important to realize that the value of $\L(\gamma,\gammat)$ depends on the choice of embedding map $F$. This is illustrated by Figure \ref{goodbadcoords}.
Nevertheless, $\L(\gamma,\gammat)$ is a topological invariant for any (fixed) choice of $F$, under deformations that vanish on $S_0$ and $S_1$.

\subsection{Field line winding of a vector field}

In principle, we could characterise the field line topology of ${\bf v}$ by the set of all pairwise winding numbers $\L(\gamma,\gammat)$ between all pairs of preimage field lines. However, we will show in Section \ref{sec:top} that this would entail significant redundancy. A more succinct description of the field line topology is given by a quantity we call the field line winding. This is defined for each field line $\gamma$, and is simply the average winding number of $\gamma$ with all other field lines in $\mathcal{C}$, \textit{i.e.},
\begin{equation}
\Lv(\gamma) = \int_{D_0}\L\big(\gamma,\gammat(x)\big)\,\mathrm{d}^2 x.
\label{eqn:lvg}
\end{equation}
Here $\gammat(x)$ denotes the field line starting from a point  $x\in D_0$. Since every field line $\gamma$ passes through $D_0$, $\Lv$ defines a scalar distribution on $D_0$. An example of such a distribution is shown in Figure \ref{fieldlinewindfigs}(c). 

By virtue of its definition, $\Lv$ is invariant under end vanishing isotopies of ${\bf v}$ that vanish on $S_0$, $S_1$. Like $\L(\gamma,\gammat)$, the field line winding $\Lv(\gamma)$ is dependent on the choice of mapping from $\mathcal{C}$ to $\mathcal{M}$; we will shortly fix this choice and hence the definition of  $\mathcal{L}$.

An alternative way to write (\ref{eqn:lvg}) is to use Cartesian coordinates $(x^1,x^2,z)$ on $\mathcal{C}$, then (\ref{eqn:lgg}) gives
\begin{equation}
\nonumber \Lv(\gamma)=\frac{1}{2\pi}\int_{0}^{1}\int_{D_z}\frac{(x^1-\gamma_1)\frac{\mathrm{d}}{\mathrm{d}z}[\gammat_2(x) - \gamma_2] - (x^2-\gamma_2)\frac{\mathrm{d}}{\mathrm{d}z}[\gammat_1(x) - \gamma_1]}{(x^1-\gamma_1)^2 +  (x^2-\gamma_2)^2}\,\mathrm{d}^2 x\rmd z,
\label{eqn:lvalt}
\end{equation}
where $\gammat(x)$ is the curve passing through the point $x=(x^1,x^2,z)\in D_z$. 

\section{Choice of embedding map} \label{sec:coords}

For measuring field line winding on a fixed subdomain $\mathcal{M}$, any choice of embedding map $F:\mathcal{C}\to\mathcal{M}$ will generate a topological measure $\Lv$ that is invariant under end-vanishing isotopies.  However, as we saw in Figure \ref{goodbadcoords}, the actual value of $\Lv$ for each field line depends on the chosen mapping. In order to allow for comparison between vector fields on different $\mathcal{M}$, it is useful to define $F$ uniquely and fix an absolute $\Lv$. In this section, we propose to do this by identifying the ``least distorted'' vector field on $\mathcal{M}$, that follows the shape of the tube as simply as possible. We then define $F$ so that the field lines of this least distorted field are the images of vertical lines on $\mathcal{C}$. This way, they will all have $\Lv\equiv0$. Vector fields with non-trivial field line topology will then have $\Lv\neq 0$, at least for a subset of their field lines.

\subsection{The least distorted field} \label{sec:uv}

A natural candidate for least-distorted braided vector field on $\mathcal{M}$ is a harmonic vector field ${\bf u}=\nabla\phi$ that satisfies
\begin{eqnarray}
&\nabla^2\phi = 0 \quad \textrm{in $\mathcal{M}$},\label{eqn:laplace}\\
&\hat{\bf z}\times\nabla\phi = 0 \quad \textrm{on $S_0$, $S_1$},\label{eqn:nxu}\\
&\hat{\bf n}\cdot\nabla\phi = 0 \quad \textrm{on $S_s$}.\label{eqn:side}
\end{eqnarray}
Here we reiterate that $\hat{\bf z}$ is the unit normal to the cross-sectional surfaces $S_z$, and $\hat{\bf n}$ is the unit normal to the side boundary $S_s$. Conditions (\ref{eqn:laplace})--(\ref{eqn:side}) define ${\bf u}$ uniquely up to magnitude, disregarding the trivial case of constant $\phi$.  Condition (\ref{eqn:nxu}) implies that $\phi$ is constant on each of the end caps, so that the integral curves are normal to these end caps. The actual values of these two constants control the magnitude and sign of ${\bf u}$, but not the shape of its integral curves. That this field can be assumed to exist on $\mathcal{M}$ is a result of Theorem 1.1 of \citeasnoun{gol2011hodge}, which provides a mixed boundary condition Hodge decomposition on Lipschitz domains. In particular the dimension of the harmonic field space, given by (1.6) in that paper, is $1$ for a cylinder; this dimension corresponds to the magnitude of the field ${\bf u}$.

To motivate the choice of condition (\ref{eqn:nxu}), consider any other harmonic field $\nabla \phi' = \nabla(\phi + \psi)$ in $\mathcal{M}$ that satisfies (\ref{eqn:side}) and has the same flux as $\nabla\phi$ through $S_0$ and $S_1$, meaning that
\begin{equation}
\int_{S_0}\hat{\bf z}\cdot\nabla \psi\,\mathrm{d}S =\int_{S_1}\hat{\bf z}\cdot\nabla \psi\,\mathrm{d}S = 0.
\end{equation}
It follows from these boundary conditions, and from the constancy of $\phi$ on $S_0$ and $S_1$, that the Dirichlet (stretching) energy of $\nabla\phi'$ is the orthogonal sum
\begin{equation}
\int_{\mathcal{M}}|\nabla\phi'|^2\,\mathrm{d}V = \int_{\mathcal{M}}|\nabla\phi|^2\,\mathrm{d}V + \int_{\mathcal{M}}|\nabla \psi|^2\,\mathrm{d}V.
\end{equation}
Therefore $\nabla\phi$ is the harmonic field in $\mathcal{M}$ that minimizes this energy for a given axial flux. In this sense, it has the least distorted integral curves. This is related to the notion of harmonic coordinates in Riemannian geometry \cite{deturck1981some}, which may be thought of as minimizing the Dirichlet energy of a coordinate map from $\mathcal{M}$ to $\mathbb{R}^3$.

\begin{figure}
\centering
\includegraphics[width=10cm]{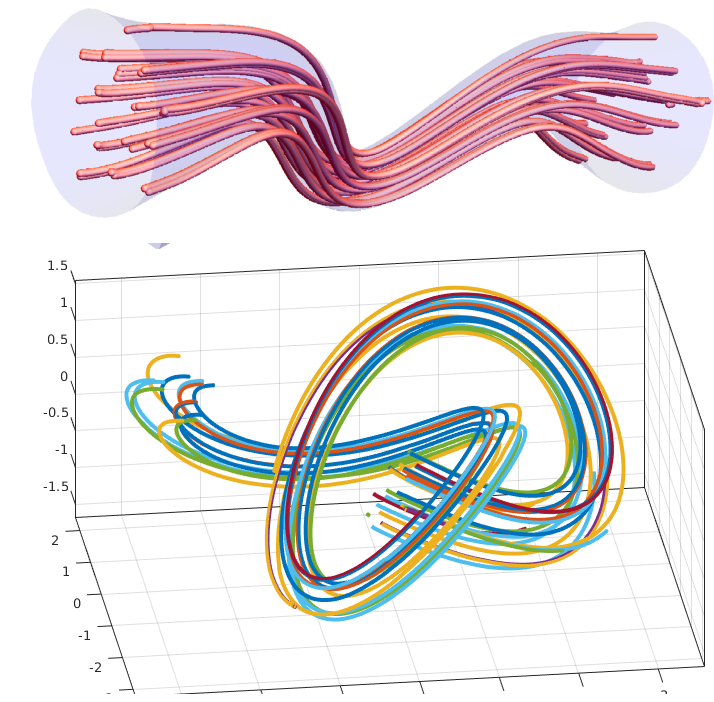}
\caption{\label{exampleu}Example domains and computations of their minimally distorted fields ${\bf u}$.}
\end{figure}

Examples of this field ${\bf u}$ are shown in Figure \ref{exampleu} for various domains, computed using the finite-element code IFEM2 in MATLAB \cite{chen2009integrated}. Note that the curves of this least distorted field flow through the domain, contouring to its shape and expanding (contracting) when the tube does. They are the natural analogue of straight lines in a Cartesian domain.

\subsection{Definition of the embedding map}

It is most convenient to specify $F^{-1}$ rather than $F$ itself directly. Our fundamental idea is that field lines of ${\bf u}$ should map to vertical lines in $\mathcal{C}$, but this leaves considerable freedom in the choice of $F^{-1}$, both in the $z$ coordinate (effectively distance along the field lines) and in which field line of ${\bf u}$ maps to which vertical line in $\mathcal{C}$.

To set the $z$ coordinate, note that ${\bf u}=\nabla\phi$ naturally defines a foliation $\{S_z\}$ of $\mathcal{M}$ by taking each $S_\phi$ to be the surface $\phi=\textrm{constant}$. Given the freedom in scaling the magnitude of ${\bf u}$, we may always arrange that $\phi=0$ on $S_0$ and $\phi=1$ on $S_1$. The fact that $\phi$ gives a valid foliation relies on the fact that ${\bf u}$ is non-zero everywhere on the interior of $\mathcal{M}$ (see \ref{leastdistortedproof}).

Choosing which field line of ${\bf u}$ maps to which vertical line is equivalent to choosing $F^{-1}$ on one of the end caps, say $S_0$. Equivalently, we must choose functions $(x^1(y), x^2(y))$ for $y\in S_0$.
To avoid measuring any spurious pairwise winding (cf. Figure \ref{goodbadcoords}), we require each of the functions $x^1$ and $x^2$ to be harmonic in the two-dimensional surface $S_0$, that is
\begin{equation}
\nabla^2_{S_0}x^{1}|_{S_0} = \nabla^2_{S_0}x^{2}|_{S_0}= 0. 
\end{equation}
where $\nabla^2_{S_0}$ indicates the Laplacian operator in the surface $S_0$. To specify a unique solution we also need to impose boundary conditions on $\partial S_0$. In order to ensure that $F^{-1}$ (and hence $F$) is one-to-one, these boundary conditions must take the form $x^{1}= \cos(\theta(s)),x^{2}=\sin(\theta(s))$, where $\theta(s)$ is an angle function of arclength $s$ along $\partial S_0$. An example is shown in Figure \ref{transcoords}, where we make the simple choice $\theta(s)=2\pi s/L$ for some arbitrary point $s=0$, where $L$ is the perimeter length of $S_0$.

\begin{figure}
    \centering
    \subfloat[]{\includegraphics[width=7cm]{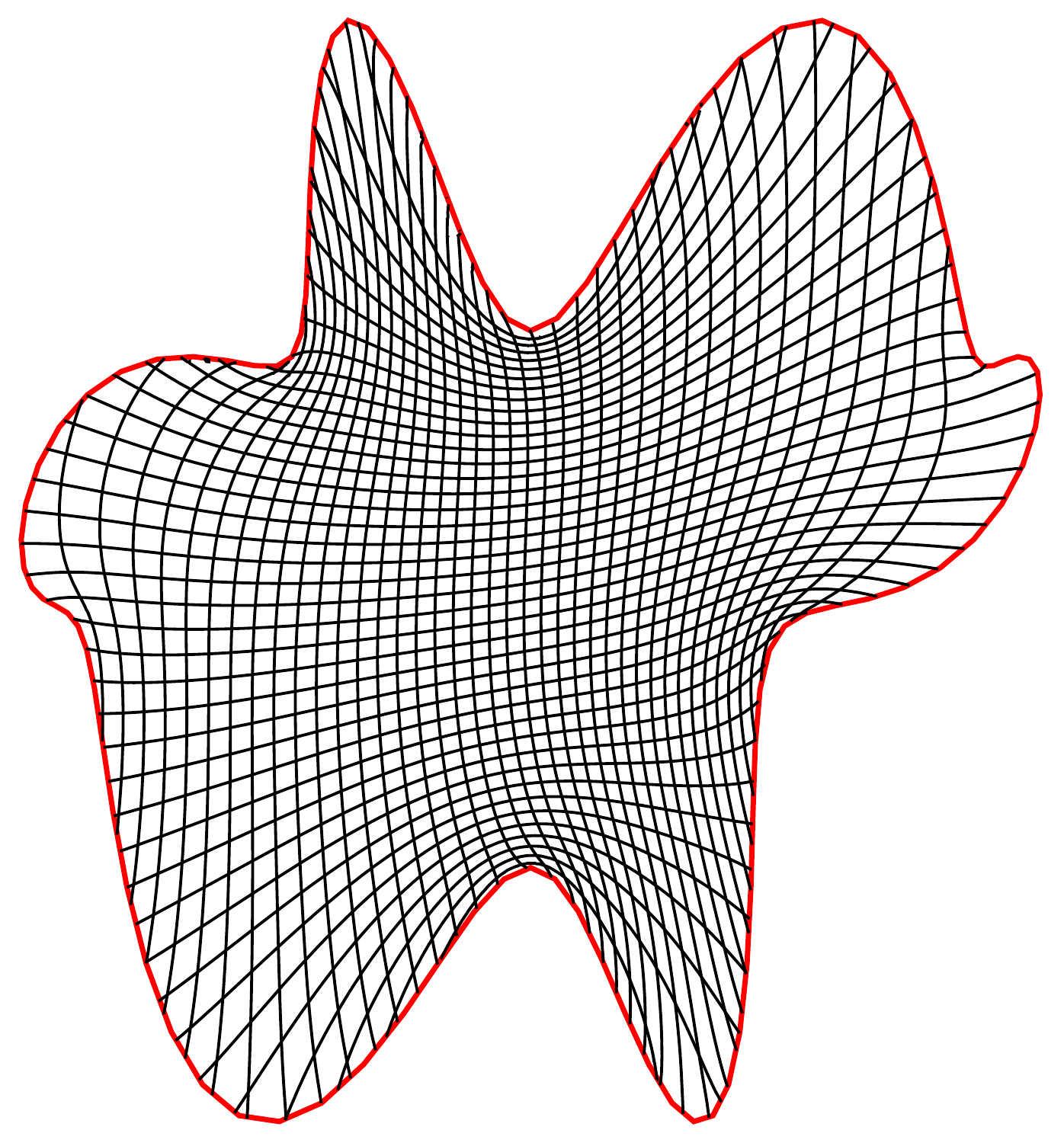}}\quad \subfloat[]{\includegraphics[width=7cm]{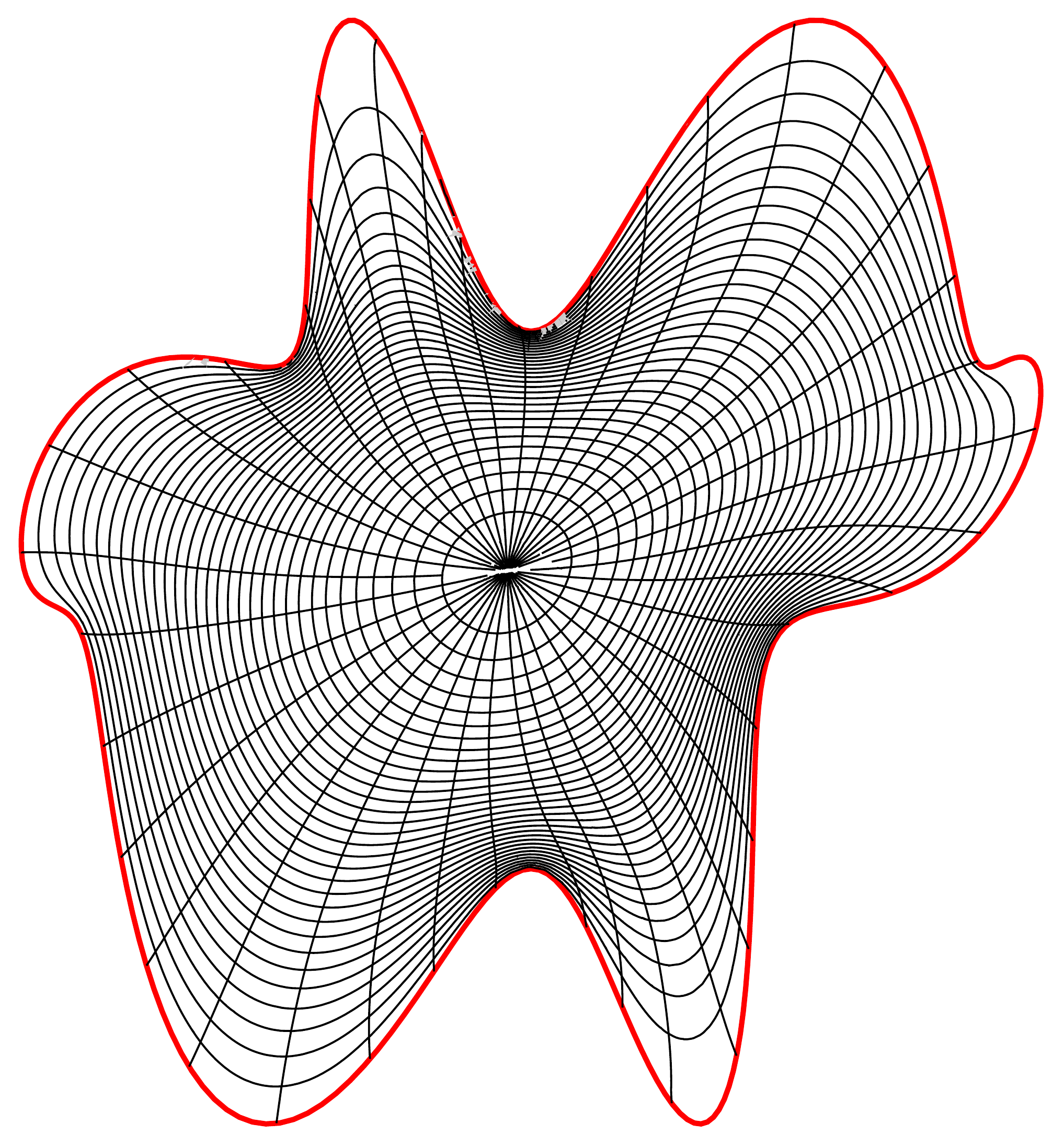}}
    \caption{Example illustrating our proposed definition of the mapping $F$ on the lower boundary, $S_0$. Black lines in (a) show the images under $F$ of $x^1$ and $x^2$ contours, whereas those in (b) show the images of polar coordinate $(r,\theta)$ contours under the same $F$.}
    \label{transcoords}
\end{figure}

\section{Topological classification} \label{sec:top}

In this section, we will show that the field-line winding $\Lv$ distribution for a braided vector field ${\bf v}$ completely determines the topology of its field lines, in the following sense.
\begin{thm}
Let ${\bf v}$, ${\bf v}'$ be two braided vector fields on the same domain $\mathcal{M}$ whose field lines on  $S_s$ are linked by an end-vanishing isotopy. Then the field lines of ${\bf v}$ and those of ${\bf v}'$ within $\mathcal{M}$ can be linked by an end-vanishing isotopy if and only if  $\Lv= \Lvp$ on all of $D_0$. 
\label{thm:top}
\end{thm}
Notice that this is purely a result about the field line curves: the magnitudes of ${\bf v}$ and ${\bf v}'$ do not matter, because $\Lv$ and $\Lvp$ depend only on the geometry of the curves. The boundary requirement on $S_s$ relates to a fundamental fact regarding the classification group of parametrized diffeomorphisms of the unit disc \cite{aref1984stirring,birman2016braids}. Essentially, if the corresponding field line mappings $f(1)$, $f'(1)$ differ on the side boundary, then this difference could be compensated by an opposite rotation on the interior, leading to both fields having the same $\Lv$.

\begin{proof}[Proof of Theorem \ref{thm:top}]
One direction is immediate: we know that $\Lv$ is invariant under any isotopy of ${\bf v}$ that vanishes on $S_0$ and $S_1$. The other direction is a deeper result. Our strategy will be to first show that equality of $\Lv$ and $\Lvp$ implies equality of the corresponding field line mappings $f(1)$, $f(1)'$ from $D_0$ to $D_1$. Then, we will show that this implies the existence of an end vanishing isotopy between the field lines of ${\bf v}$ and those of ${\bf v}'$. 
 
To show equality of the field line mappings $f(1)$ and $f'(1)$, consider the gradient of $\Lv(\gamma)$ with respect to the field line startpoint $(r,\theta)\in D_0$. This is most succinctly expressed in the language of differential forms as $\mathrm{d}_\perp\Lv$, and we claim that
\begin{equation}
    \mathrm{d}_\perp\Lv = f^*\alpha - \alpha, \qquad \textrm{where} \quad \alpha(r,\theta,z) = \frac{r^2}{2}\mathrm{d}\theta.
\label{eqn:dlvc}
\end{equation}
To see this, note that we can write $\Lv$ in terms of the field line mapping $f$ as
\begin{equation}
    \Lv(\gamma) = \int_{D_0}\Big[ \frac{1}{2\pi}\Big(f^*\Theta - \Theta\Big) + N(\gamma,\gammat(x))\Big]\widetilde{r}\,\mathrm{d}^2x,
\end{equation}
where $\Theta(\gamma,\gammat(x),z)$  is the angle made in the surface $D_z$ by the curve $\gamma$ and another curve $\gammat$ rooted at a point $x\in D_0$. Treating this as a function of the $\gamma$-startpoint $(r,\theta)\in D_0$, we differentiate to find
\begin{eqnarray}
\mathrm{d}_\perp\Lv &= \frac{1}{2\pi}\int_{D_0}\Big(f^*\mathrm{d}_\perp \Theta - \mathrm{d}_\perp\Theta\Big)\,\mathrm{d}^2x.
\end{eqnarray}
One may show by explicit calculation (see \ref{app:rtheta}) that
\begin{equation}
    \frac{1}{2\pi} \int_{D_0}\mathrm{d}_\perp\Theta\,\mathrm{d}^2x= \alpha,
    \label{eqn:rth}
\end{equation}
from which (\ref{eqn:dlvc}) follows.

Now consider the composite map $g:D_0\to D_0$, defined by $g=f(1)^{-1}\circ f'(1)$. Under our assumption $\Lv=\Lvp$, it follows from equation (\ref{eqn:dlvc}) that $g^*\alpha=\alpha$ \cite{yeates2013unique}. Equivalently, in components,
\begin{eqnarray}
\frac{g_r^2}{2}\frac{\partial g_\theta}{\partial r} = 0, \qquad
\frac{g_r^2}{2}\frac{\partial g_\theta}{\partial\theta} = \frac{r^2}{2}.
\end{eqnarray}
The first equation shows that $g_\theta(r,\theta) = G(\theta)$, but the fact that $f(1)=f'(1)$ on the side boundary  (by our assumption on $S_s$) means that $g_\theta(r,\theta)=\theta$. The second equation then shows that $g_r(r,\theta)=r$, so $g$ is the identity and hence $f(1)=f'(1)$.

This leaves us to show that there exists an isotopy between any two braided vector fields with the same field line mapping $f(1)$. For this we utilise some established results. The field lines $f(z)$ are parametrised curves in the group of diffeomorphisms of the unit disc, $\mathrm{Diffeo}(D^2)$. Since $f(0)=f'(0)=\mathrm{id}$ and $f(1)=f'(1)$, the composite map
\begin{equation}
f_c(z) = \left\{
\begin{array}{cc}
f(z) & z\in[0,1)\\
f'(2-z) & z\in[1,2].
\end{array}
\right.
\end{equation} 
forms a closed curve in this group which is anchored at the identity. If this curve can be shrunk to the identity within the group by a homotopy then the field lines can be linked by an isotopy. Thus we are interested in the homotopy group $\pi_1(\mathrm{Diffeo}(D^2))$ of curves in $\mathrm{Diffeo}(D^2)$ (anchored at the identity). A number of established facts related to this group allow us to answer the question. Firstly, there is a restriction map  
\begin{equation}
\mathrm{Diffeo}(D^2) \rightarrow \mathrm{Diffeo}(S^1),
\end{equation}
obtained by restricting each diffeomorphism in $D^2$ to the boundary $\partial D^2=S^1$.  This is a fibration
with fiber $\mathrm{Diffeo}_{\partial D^2}(D^2)$, namely diffeomorphisms of the disc that fix the boundary.
This implies that the map $\pi_1(\mathrm{Diffeo}((D^2))\rightarrow \pi_1(\mathrm{Diffeo}(S^1))$ is an isomorphism -- see \citeasnoun{mann2013diffeomorphism} in conjunction with Theorem 4.41 in \citeasnoun{hatcher2002algebraic}. Finally there is a homotopy equivalence between $\mathrm{Diffeo}(S^1)$ and the group $O^2$ of rotations \citeaffixed{smale1959diffeomorphisms}{see}.
Putting this information together implies that the set of closed curves in the space $\mathrm{Diffeo}(D^2)$, from which our map $f_c$ is drawn, can be categorised up to homotopy by their behaviour when restricted to the boundary of the cylinder. Then the fact that the homotopy group has the structure of $O^2$ implies that closed curves in $\mathrm{Diffeo}(D^2)$ are homotopic to the identity if and only if they describe no net integer rotation on the boundary. The boundary assumption in the theorem ensures this for $f_c$.
\end{proof}

\section{Relation to helicity}\label{sec:helicity}

Since each pairwise winding number $\L(\gamma,\gammat)$ is itself invariant under end-vanishing isotopies of ${\bf v}$, it is clear that any functional
\begin{equation}
    W_{\bf v}(\gamma) = \int_{D_0}w(x)\L(\gamma,\gammat(x))\,\mathrm{d}^2x
\end{equation}
will be similarly invariant. The field line winding $\Lv(\gamma)$, with $w\equiv 1$ is the simplest such invariant, and is already sufficient to completely determine the field line topology of ${\bf v}$. Nevertheless, there may be situations in which other weightings are of physical interest.

The most notable example is field line helicity, which is defined for the subset of divergence-free braided vector fields. Typical examples are magnetic fields (in magnetohydrodynamics) or vorticity fields. In this section, we shall denote such fields as ${\bf b}$ rather than ${\bf v}$. The field line helicity is the particular $W_{\bf v}$ invariant given by
\begin{equation}
    A_{\bf b}(\gamma) = \int_{D_0}J_0(x)b_z(x)\,\L(\gamma,\gammat(y))\,\mathrm{d}^2x,
    \label{eqn:flh}
\end{equation}
where $J_0$ is the Jacobian of the map $F^{-1}:S_0\to D_0$, and $b_z$ is the normal component of ${\bf b}$ on $S_0$. When $\mathcal{M}=\mathcal{C}$, we have $J_0\equiv 1$ so 
\begin{equation}
    A_{\bf b}(\gamma) = \int_{D_0}b_z(x)\,\L(\gamma,\gammat(x))\,\mathrm{d}^2x.
    \label{eqn:flh1}
\end{equation}
For this case, \citeasnoun{prior2014helicity} showed that (\ref{eqn:flh1}) is equivalent to the original definition of field line helicity, $A_{\bf b}(\gamma) = \int_\gamma{\bf a}\cdot\,\mathrm{d}{\bf l}$ (where ${\bf b}=\nabla\times{\bf a}$) provided that the vector potential ${\bf a}$ is in the so-called winding gauge
\begin{equation}
    {\bf a}(x^1, x^2, z) = \frac{1}{2\pi}\int_{D_z}\frac{{\bf b}(\tilde{x}^1,\tilde{x}^2,z)\times(x^1-\tilde{x}^1,x^2-\tilde{x}^2,0)}{(x^1-\tilde{x}^1)^2 + (x^2-\tilde{x}^2)^2}\,\mathrm{d}^2\tilde{x}.
    \label{eqn:winding}
\end{equation}
Indeed, they used the geometrical interpretation in terms of winding numbers to motivate this particular choice of gauge.

For more general $\mathcal{M}$, we can similarly relate (\ref{eqn:flh}) to a particular choice of vector potential. It is convenient to think of ${\bf b}$ as a differential 2-form $\beta = \mathrm{i}_{\bf b}\mathrm{vol}^3$ on $\mathcal{M}$, where $\mathrm{i}$ is the interior product and $\mathrm{vol}^3$ the volume form \cite{frankel2011geometry}. Then pushing forward the integral from $D_0$ to $S_0$, and using equation (3.13) of \citeasnoun{frankel2011geometry}, we have
\begin{eqnarray}
    A_{\bf b}(\gamma) &= \int_{S_0}\beta_z(\tilde{y})\,\L(\gamma,\gammat(\tilde{y}))\,\mathrm{d}\tilde{y}^1\wedge\mathrm{d}\tilde{y}^2\\
    &= \int_0^1\int_{S_0}\beta_z(\tilde{y})\frac{\mathrm{d}}{\mathrm{d}z}\Theta(\gamma, \gammat(\tilde{y}), z)\,\mathrm{d}\tilde{y}^1\wedge\mathrm{d}\tilde{y}^2\,\mathrm{d}z
\end{eqnarray}
Now, the divergence free property of ${\bf b}$ means that the form $\beta_z\,\mathrm{d}y^1\wedge\mathrm{d}y^2$ is conserved along field lines, so we can write
\begin{equation}
A_{\bf b}(\gamma) = \int_0^1\int_{S_z}\beta_z(\tilde{y})\frac{\mathrm{d}}{\mathrm{d}z}\Theta(\gamma, \gammat(\tilde{y}), z)\,\mathrm{d}\tilde{y}^1\wedge\mathrm{d}\tilde{y}^2\,\mathrm{d}z.
\end{equation}
Using (\ref{eqn:lgg}) and the fact that $\mathrm{d}\gamma_i/\mathrm{d}z = \beta_i/\beta_z$, we find that 
\begin{equation}
    A_{\bf b}(\gamma) = \int_\gamma\alpha,
\end{equation}
where the one-form $\alpha=\alpha_1\mathrm{d}y^1 + \alpha_2\mathrm{d}y^2 + \alpha_z\mathrm{d}z$ has components
\begin{eqnarray}
\alpha_1(y) &= -\frac{\beta_z(\tilde{y})(y^2-\tilde{y}^2)}{(y^1-\tilde{y}^1)^2 + (y^2-\tilde{y}^2)^2}\,\mathrm{d}\tilde{y}^1\wedge\mathrm{d}\tilde{y}^2,\\
\alpha_2(y) &= \frac{\beta_z(\tilde{y})(y^1-\tilde{y}^1)}{(y^1-\tilde{y}^1)^2 + (y^2-\tilde{y}^2)^2}\,\mathrm{d}\tilde{y}^1\wedge\mathrm{d}\tilde{y}^2,\\
\alpha_z(y) &= \frac{\beta_1(\tilde{y})(y^2-\tilde{y}^2) - \beta_2(y^1-\tilde{y}^1)}{(y^1-\tilde{y}^1)^2 + (y^2-\tilde{y}^2)^2}\,\mathrm{d}\tilde{y}^1\wedge\mathrm{d}\tilde{y}^2.
\end{eqnarray}
This one-form corresponds to a vector potential ${\bf a}$ on $\mathcal{M}$ which is the generalisation of the winding-gauge vector potential (\ref{eqn:winding}). 

Notice that $A_{\bf b}$ contains no additional topological information compared to $\Lb$; the difference is that it weights each field line differently according to $b_z$. This weighting has physical meaning because  $\beta_z\,\mathrm{d}y^1\wedge\mathrm{d}y^2$ is preserved along field lines -- thus the weighting is a property of the field line as a whole, not just an arbitrary function of the endpoints on $S_0$. An interesting consequence of the additional weighting is that, although $A_{\bf b}$ is invariant under end-vanishing isotopies of ${\bf b}$ like $\Lb$, it may behave differently from $\Lb$ if the domain $\mathcal{M}$ changes shape. This remains to be investigated further.

\section{Conclusions and comparison to other work} \label{sec:conclusion}

The primary development in this work (Theorem \ref{thm:top}) is the proof of a topological classification theorem for braided vector fields on tubular subdomains, which are homeomorphic to the cylinder as defined in Section \ref{sec:coords}.
Specifically, the field lines of a given braided field may be continuously deformed into those of another other field -- without moving their end-points -- if and only if the two fields have the same distribution of field line winding, $\Lv$. There is a small caveat: this invariant cannot distinguish overall full (integer) rotations. These could be detected solely from the field line mappings on the side boundary, and are unlikely to be problematic in practical applications. Our result strengthens an earlier theorem of \citeasnoun{yeates2014complete} and extends it to a broader class of vector fields; specifically to include those which have divergence. The measure $\Lv$ has recently been applied to study reconnection in a laboratory plasma experiment \cite{prior2018quantifying}; here we have formalised the underpinning result and extended it to tubular sub-domains that cannot be foliated by parallel planes, including tubular domains foliated by non-Euclidean surfaces.

The unique definition of $\Lv$ relies on the identification of a least distorted vector field ${\bf u}$, defined in Section \ref{sec:uv},  whose integral curves are mapped to straight lines on the cylinder (on which the field line winding is defined). Our choice of ${\bf u}$ and its mapping to the cylinder have three critical properties:
\begin{enumerate}
\item{The vector field ${\bf u}$ minimises the Dirichlet (distortion) energy amongst all braided vector fields on a given tubular domain $\mathcal{M}$ with a specified flux, \textit{i.e.} it minimises
\begin{equation}
E= \int_{\mathcal{M}}{\bf v}\cdot{\bf v}\,\rmd V,  \quad \textrm{for specified } \int_{S_0}\hat{\bf z}\cdot{\bf v}\,\rmd{S}.
\end{equation}
The latter constraint relates to the fact the energy varies with flux for a fixed topology, whereas we are interested purely in finding the vector field of least entangled topology. 
}\item{
In this mapping, all field lines of ${\bf u}$ have zero pairwise winding number, so the winding measure $\Lv$ is everywhere zero for ${\bf u}$ itself.
}\item{To determine which field line of ${\bf u}$ maps to which vertical line on $\mathcal{C}$, we use harmonic coordinates so as to avoid the erroneous measurement of entanglement such as in Figure \ref{goodbadcoords}.} 
\end{enumerate}
The combination of these three properties means that a non-trivial distribution $\Lv$ necessarily indicates non-trivial internal entanglement of the original vector field ${\bf v}$. This entanglement may be seen as increasing the energy above the minimum level enforced by the distorted global shape of the domain $\mathcal{M}$ itself. The only comparable construction to this of which we are aware is the so-called Mermin-Ho basis, which is defined on a pre-determined Riemannian manifold. Using the notion of parallel transport with respect to a preferred vector field (not specifically assigned in the theory) one can define a local basis on the manifold from which any local deviations of an arbitrary vector from this field are quantified \cite{mermin1976circulation,kamien2002geometry}. Our work differs in that we define and justify the choice of a preferred vector field. Moreover, this justification is based on a globally defined notion (the winding rate $\rmd \Theta(\gamma,\gammat,z)/\rmd z$), rather than a local notion (parallel transport). The Godbillion-Vey invariant see \textit{e.g.} \cite{arnold1992topological}, can be related to the helicity of the integral curves of the foliation of a \textbf{closed} domain so is not directly comparable here (see \cite{webb2014local} for a discussion of its applicability in MHD contexts).

The invariant $\Lv$ used in our classification theorem provides a fundamental description of the field line topology, in the sense that it depends only on the integral curves of the vector field and not on its magnitude or flux. This is in contrast to the widely used (field line) helicity invariants for divergence-free fields  \cite{webb2014local,russell2015evolution,cantarella2010new,yeates2018relative,moraitis2019relative}, which do depend on the flux. Indeed, we showed in Section \ref{sec:helicity} that the field line helicity -- effectively, the helicity of an infinitesimal flux tube surrounding a single field line -- takes a similar integral form to $\Lv$ but with an additional weighting of flux. As such, it mixes a purely topological quantity (the pairwise winding) with a physical quantity (the flux). This indicated the invariance of helicity under Euler flows (ideal motions) which vanish at the domain boundaries requires not only the invariance of $\L(\gamma,\gammat)$ but also the divergence-free condition, which ensures preservation of flux. Since the flux can change without altering the field line topology, $\Lv$ is revealed as a more fundamental quantity than the (field line) helicity.

\appendix
\section{Validity of the foliation defined by ${\bf u}$}\label{leastdistortedproof}
\begin{figure}
\centering
\subfloat[]{\includegraphics[width=6cm]{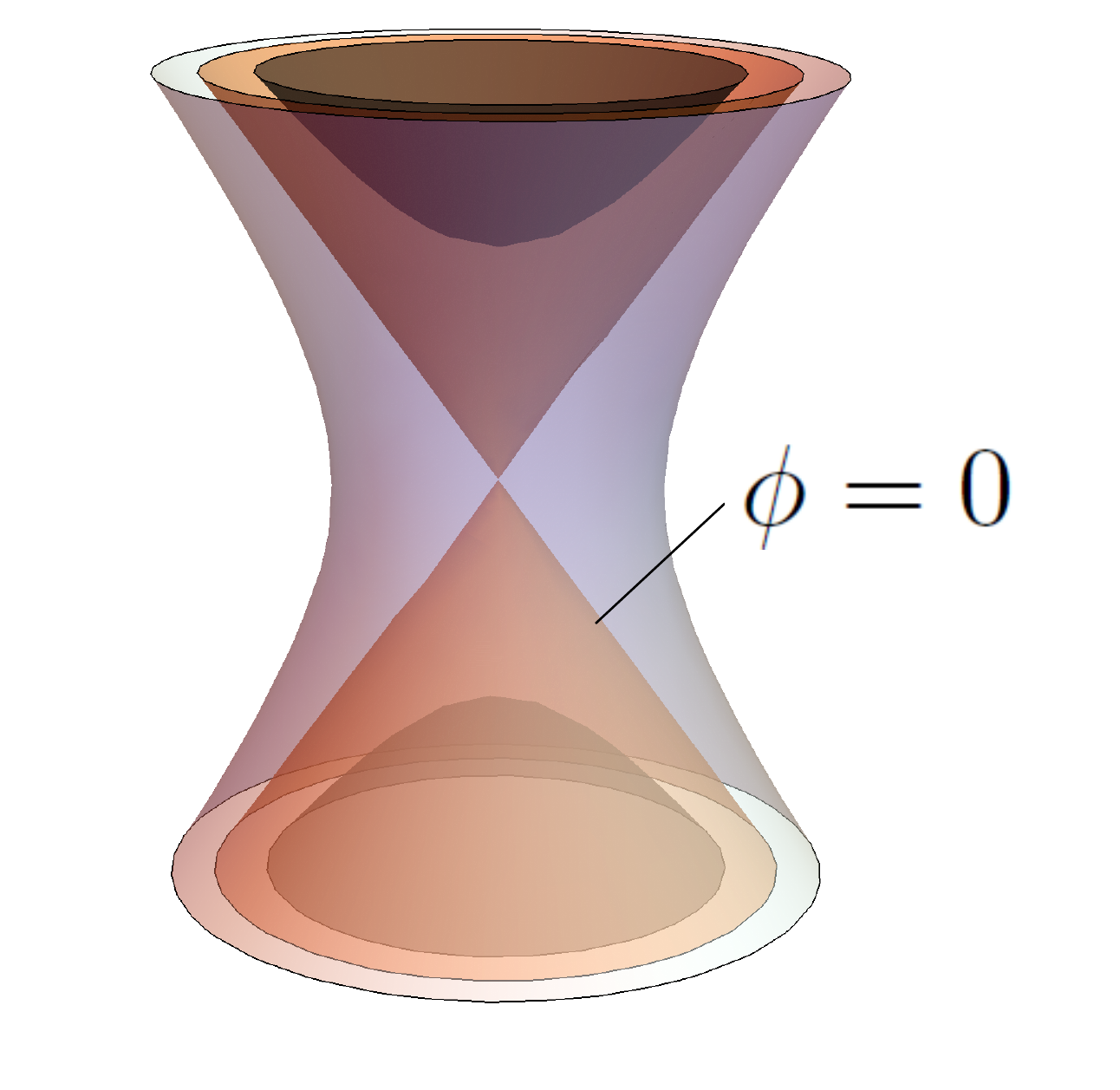}}\quad \subfloat[]{\includegraphics[width=6cm]{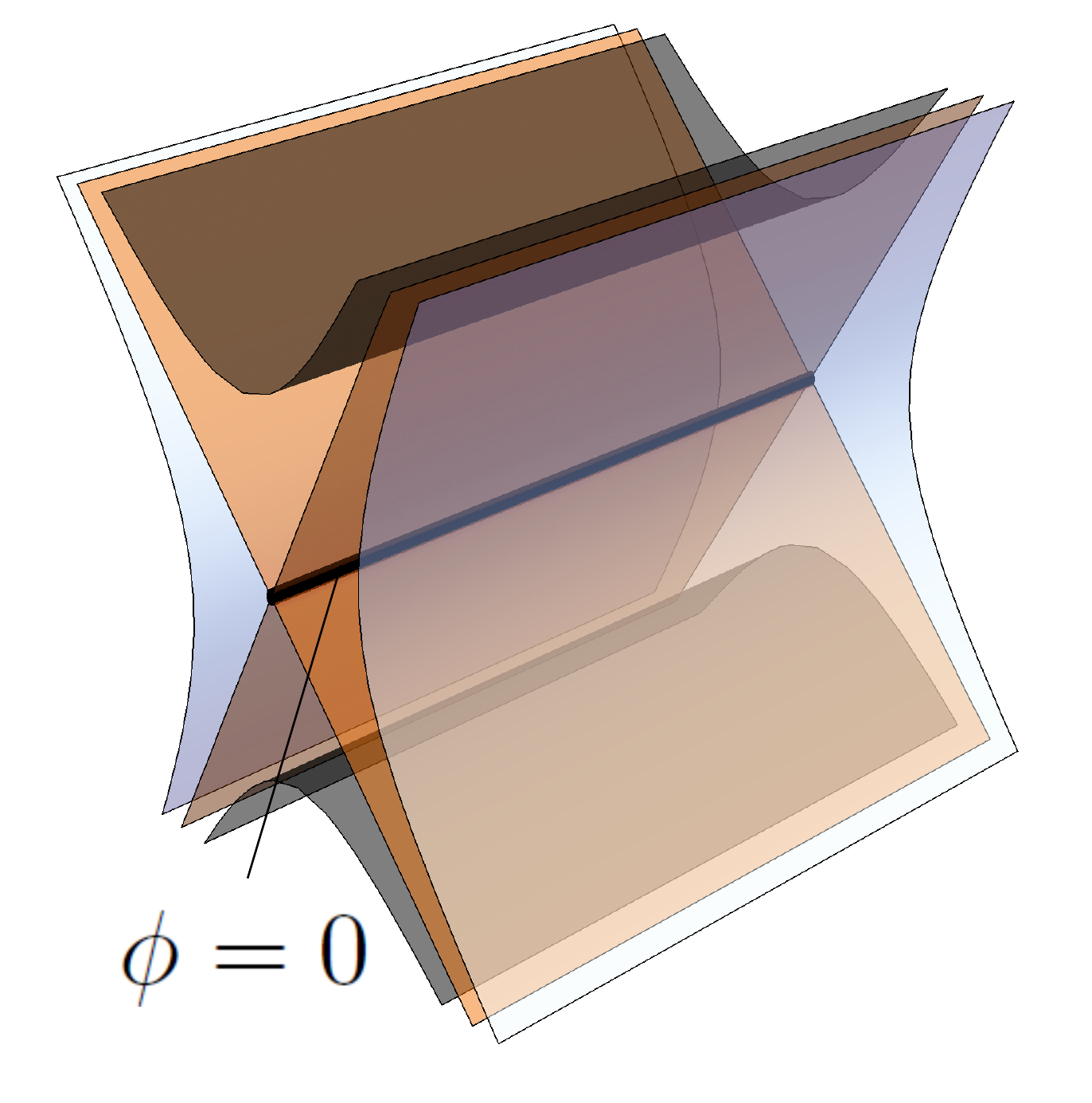}}
\caption{\label{singularstruct}Possible local structures of the constant-$\phi$ surfaces around null points of  harmonic fields ${\bf u}=\nabla\phi$. There are two possibilities, depending on whether the null point is isolated (a) or part of a null line (b).}
\end{figure}

In this appendix we show that the least-distorted harmonic field ${\bf u}$ defines a valid foliation of the domain $\mathcal{M}$. In particular, we will show that ${\bf u}\neq{\bf 0}$ everywhere on the interior of $\mathcal{M}$, so that there is a unique surface $S_z$ through all interior points and the winding integrals are well defined. It is possible to have points with ${\bf u}={\bf 0}$ on the boundary $\partial M$, but these do not prevent the existence of the foliation.

To rule out null points of our harmonic field ${\bf u}=\nabla\phi$, the idea is to consider the possible structures of constant-$\phi$ surfaces that could result from such points, and show that such surfaces cannot exist. If $p$ is a hypothetical point where ${\bf u}(p)={\bf 0}$, then the local linearisation of ${\bf u}$ around $p$ is constrained by the fact that ${\bf u}$ is harmonic. Firstly, the tensor ${\bf u}$ is symmetric due to $\nabla\times{\bf u}={\bf 0}$, meaning that it has real eigenvalues $\lambda_1$, $\lambda_2$, $\lambda_3$. Secondly, it is traceless due to $\nabla\cdot{\bf u}=0$, so that $\lambda_1+\lambda_2+\lambda_3=0$. There are two possible structures around $p$, shown in Figure \ref{singularstruct}.  If all three eigenvalues are non-zero, then in suitable local coordinates we may write
\begin{equation}
\label{pointsingularity}
{\bf u} = \nabla \phi \approx \pm \left(x,y,-2z\right),\mbox{ so } \phi \approx \pm \frac{1}{2}\left( x^2 + y^2 - 2z^2\right).
\end{equation}
Then the surface $\phi=0$ has the topology of two cones whose tips meet at $p$ (Figure \ref{singularstruct}a), separating two-sheet hyperboloids with $\phi>0$ and one-sheet hyperboloids with $\phi<0$. Alternatively, if one eigenvalue is $0$ then $p$ is part of a null line (Figure \ref{singularstruct}b) and, again in  suitable local coordinates, we may write 
\begin{equation}
{\bf u} = \nabla \phi = \pm \left(x,0,-z\right),\mbox{ so } \phi = \pm \left( x^2 - z^2\right).
\end{equation}
The $\phi=0$ surface is then a pair of intersecting planes. The null line could either be a closed loop or could intersect the boundary $\partial\mathcal{M}$. It is not possible to have a null surface, nor for all three eigenvalues to vanish at $p$ (unless ${\bf u}={\bf 0}$ on the whole domain).
%There cannot be any surface singular structure, which would require two of the eigenvalues to be zero and hence require the third be zero. It is possible that all three eigenvalues are zero, but unless the field is zero on the whole domain there would necessarily be some surface on the domain interior for which the field adopts a two dimensional singular structure (one one side) and hence which would not be permissible. 

To rule out null points and lines, we must consider how the $\phi$ surfaces extend outside the neighbourhoods of these structures. To do so, we use the fact that $\phi$ defines a continuous foliation of $\mathcal{M}$ everywhere outside the null set $\mathcal{N}=\{p\in\mathcal{M}\,|\, {\bf u}(p)={\bf 0}\}$. This follows from the Frobenius theorem.  So the $\phi=0$ surfaces from null points or lines extend smoothly away from these locations and can end only on $\mathcal{N}$ or on the boundary $\partial\mathcal{M}$. In fact, they can only intersect (transversely) the side boundary $S_s$, owing to the boundary conditions of constant $\phi$ on $S_0$ and $S_1$.

\begin{figure}
\centering
\subfloat[]{\includegraphics[width=4cm]{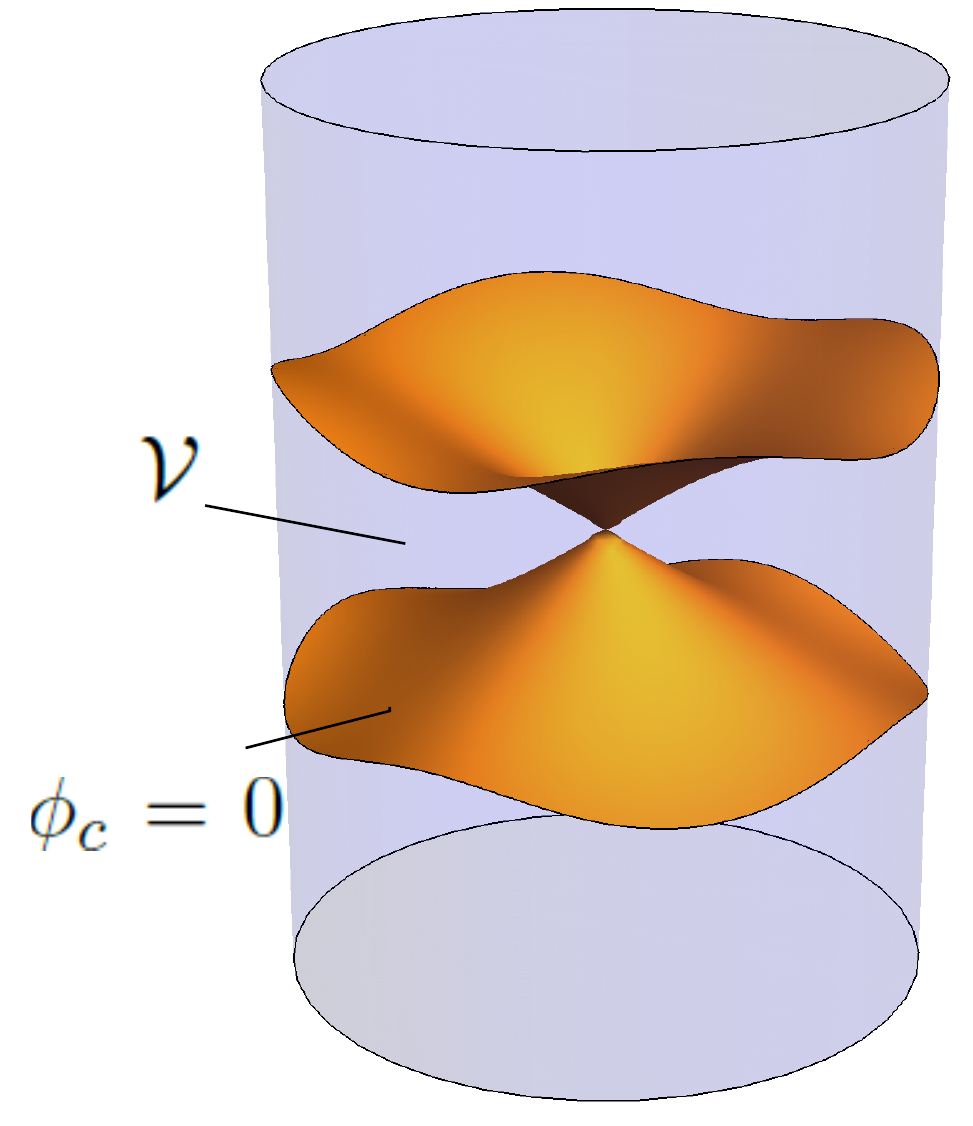}}
\subfloat[]{\includegraphics[width=4cm]{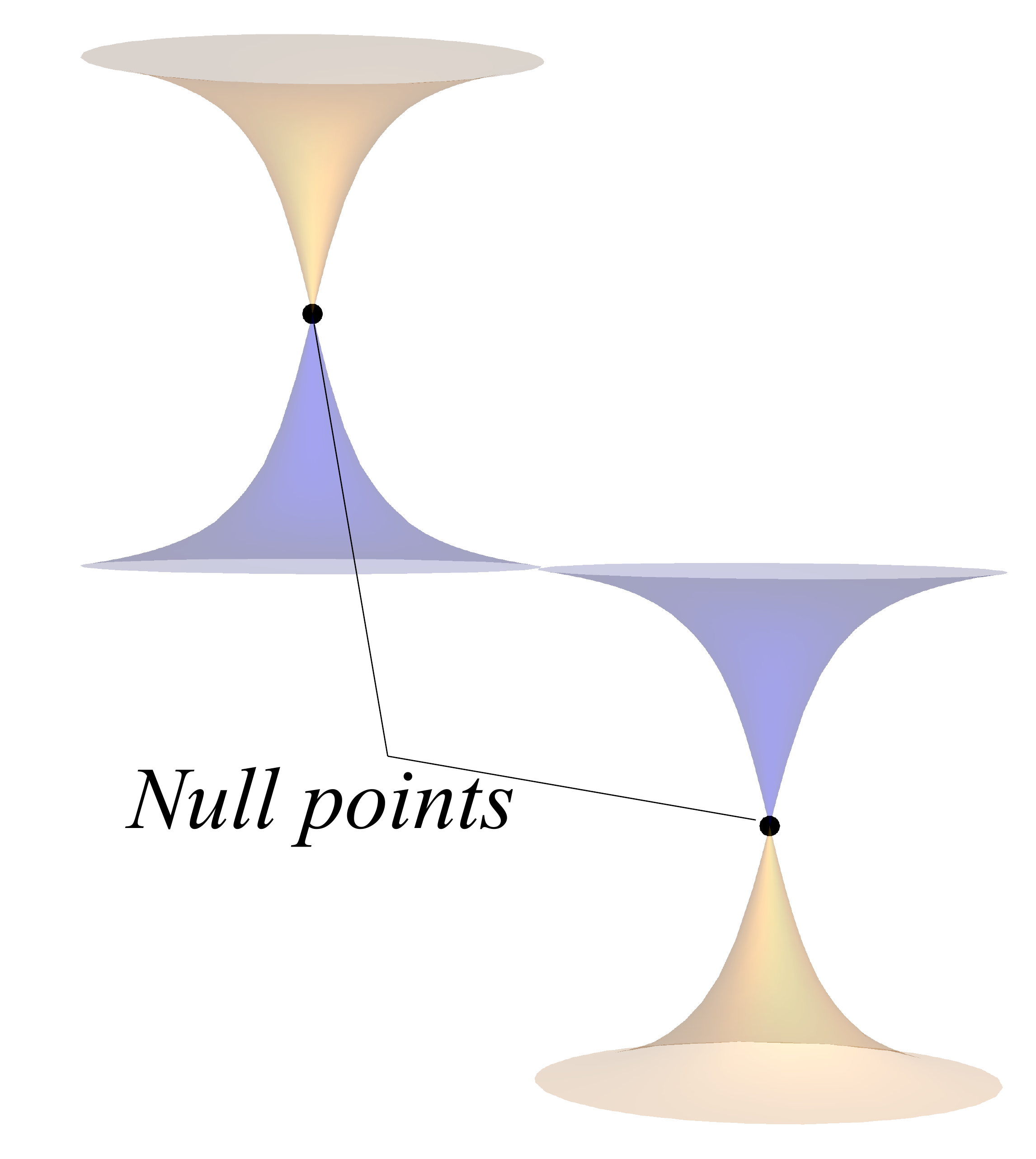}} \subfloat[]{\includegraphics[width=5cm]{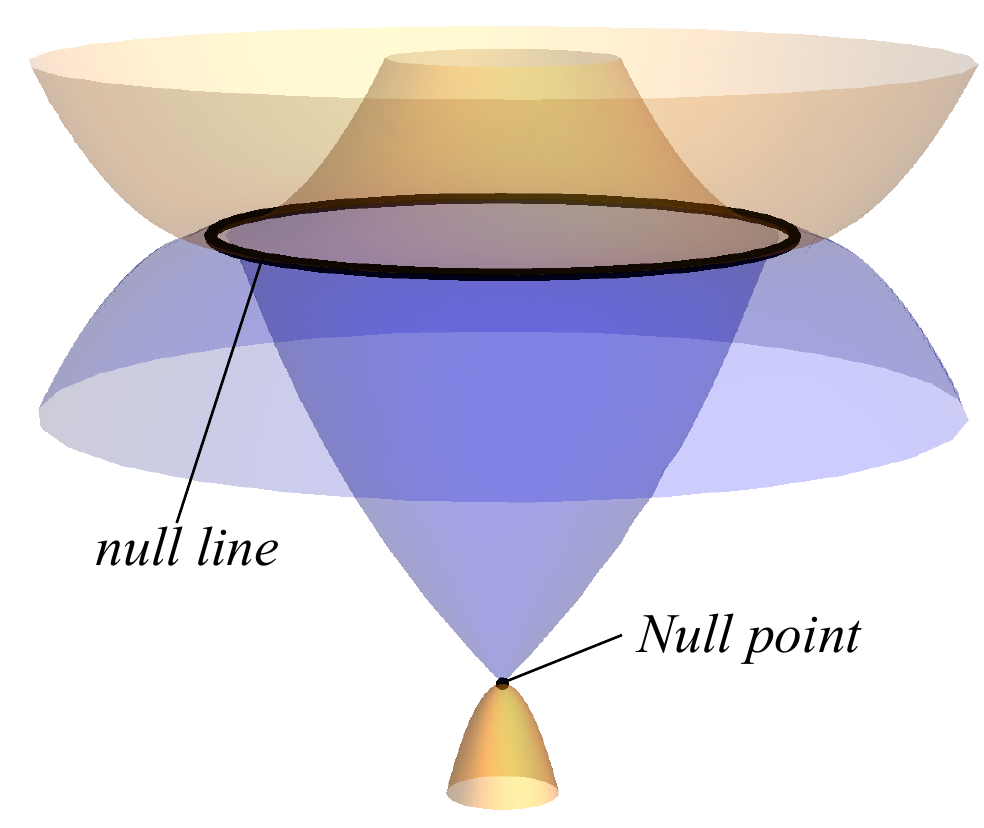}}
\subfloat[]{\includegraphics[width=3cm]{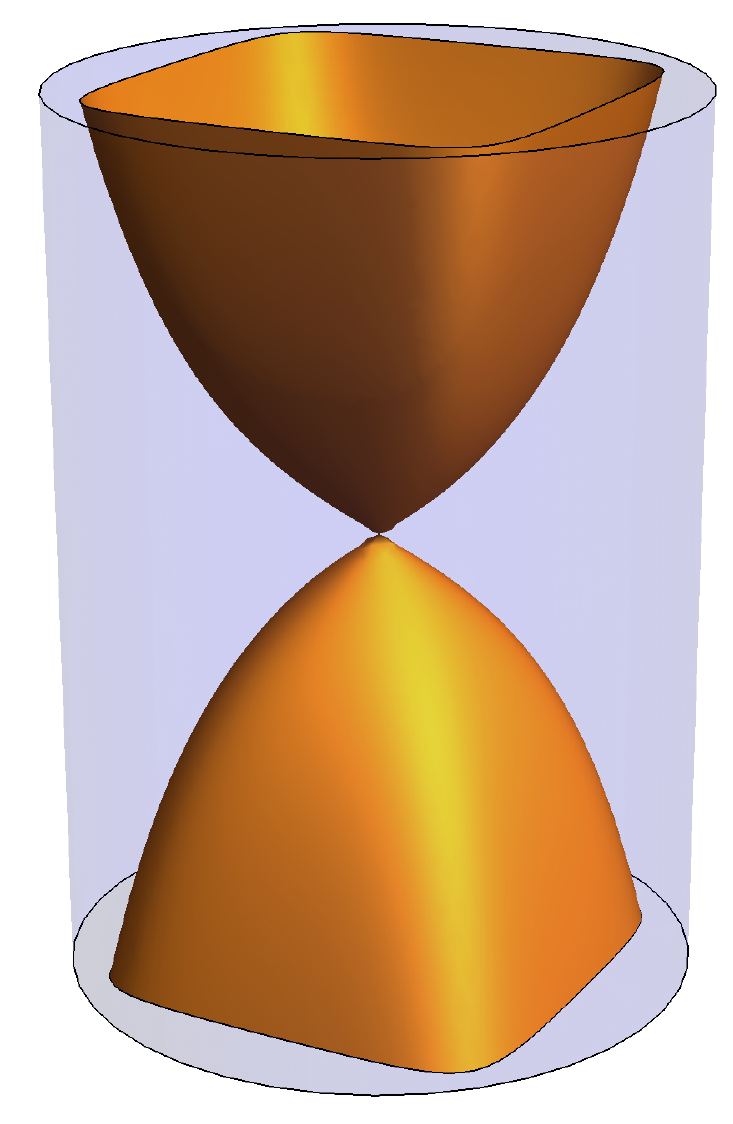}}
\caption{\label{singlesingularproof}Global extensions of the $\phi =0$ surfaces from (a) an isolated null point, (b) two isolated null points, and (c) a null point and a null line. In (a), the surfaces intersect the side boundary $S_s$, enclosing a volume $\mathcal{V}$. In (d) the $\phi=0$ surface intersects $S_0$ and $S_1$ transversely, but this is forbidden by the boundary conditions $\phi=\textrm{constant}$ on $S_0$ and $S_1$.} 
\end{figure}

Consider first a single null point in the interior, as in Figure \ref{singlesingularproof}(a). Each of the two $\phi=0$ surfaces must either be closed or end on $S_s$ (as shown in Figure \ref{singlesingularproof}a). To show that neither is possible, let $\hat{\bf n}$ be the outward normal from the indicated volume $\mathcal{V}$. Then observe that $\hat{\bf n}\cdot{\bf u}$ is non-zero with the same sign everywhere on these two surfaces where they bound $\mathcal{V}$, because $\phi < 0$ on one side and $\phi>0$ on the other. Since the two surfaces intersect $S_s$, and since $\hat{\bf n}\cdot{\bf u}=0$ on $S_s$, there must be a net flux in/out of $\mathcal{V}$, contradicting $\nabla\cdot{\bf u}=0$. If the surface were closed instead of intersecting $S_s$, we would similarly obtain a contradiction. So a single null point in the interior cannot exist.

This argument extends to rule out multiple null points or lines in the interior (of which there can only be a countable number). Tracing the $\phi=0$ surfaces from some null point or line, we can extend these surfaces, possibly incorporating further points in $\mathcal{N}$ -- as depicted in Figure \ref{singlesingularproof}(b) or (c) -- until the composite surfaces either close on themselves or intersect $S_s$. Either way we can find a finite volume enclosed with a non-zero flux of ${\bf u}$ through its boundary, which contradicts $\nabla\cdot{\bf u}=0$. 

Finally we point out that since the boundary is  Lipschitz
continuous it is possible to embed the conic structure of a null surface $\phi=0$ in the boundary. As an example this would be necessary on the bounding curves $\partial S_{0/1}$ for an hourglass domain morphology, owing to the boundary conditions on ${\bf u}$.

\section{Derivation of Equation (\ref{eqn:rth})} \label{app:rtheta}
To evaluate the left-hand side of (\ref{eqn:rth}) at the point $x_0=(r,\theta)$, consider the components of $\mathrm{d}_\perp\Theta$ in polar coordinates $(r,\theta)$. Differentiating the expression (\ref{eqn:Theta}) for $\Theta(x_0,x,z)$ gives
\begin{eqnarray}
\frac{\partial\Theta}{\partial r} &= \frac{(x_0^1-x^1)\sin\theta - (x_0^2-x^2)\cos\theta}{(x_0^1-x^1)^2 + (x_0^2-x^2)^2},\\
\frac{\partial\Theta}{\partial \theta} &= r\frac{(x_0^1-x^1)\cos\theta + (x_0^2-x^2)\sin\theta}{(x_0^1-x^1)^2 + (x_0^2-x^2)^2}.
\end{eqnarray}
To evaluate the integrals, change from $(x^1,x^2)$ to polar coordinates centered at $x_0$,
\begin{equation}
    x^1 = \rho\cos\varphi + x_0^1, \qquad x^2 = \rho\sin\varphi + x_0^2.
    \label{eqn:y}
\end{equation}
In these coordinates,
\begin{eqnarray}
\frac{\partial\Theta}{\partial r} = \frac{\sin(\varphi-\theta)}{\rho}, \qquad \frac{\partial\Theta}{\partial \theta} = -r\frac{\cos(\varphi-\theta)}{\rho}.
\end{eqnarray}
Integrating with respect to $(\rho,\varphi)$ we must account for the fact that the $\rho$ limit of integration is a function of $\varphi$. Substituting expressions (\ref{eqn:y}) into the equation $(x^1)^2 + (x^2)^2=1$ for the boundary leads to the integration limit
\begin{eqnarray}
\rho(\varphi) = -r\cos(\varphi-\theta) + \sqrt{r^2\cos^2(\varphi-\theta) - r^2 - 1}.
\end{eqnarray}
Using the phase behaviour of $\sin$ and $\cos$, we may then evaluate
\begin{eqnarray}
\frac{1}{2\pi}\int_{D_0}\frac{\partial\Theta}{\partial r}\,\rho\mathrm{d}\rho\mathrm{d}\varphi = \frac{1}{2\pi}\int_0^{2\pi}\sin(\varphi-\theta)\rho(\phi)\,\mathrm{d}\varphi = 0
\end{eqnarray}
and
\begin{eqnarray}
\frac{1}{2\pi}\int_{D_0}\frac{\partial\Theta}{\partial \theta}\,\rho\mathrm{d}\rho\mathrm{d}\varphi = \frac{r}{2\pi}\int_0^{2\pi}\cos^2(\varphi-\theta)\,\mathrm{d}\varphi = \frac{r}{2}.
\end{eqnarray}

\section*{References}
\bibliographystyle{jphysicsB}
\bibliography{topology}

\end{document}